\documentclass[12pt]{amsart}
\usepackage{hyperref}
\usepackage{young}
\usepackage{fullpage}
\usepackage[alphabetic]{amsrefs}

\newtheorem{thm}{Theorem} 
\newtheorem{cor}[thm]{Corollary}
\newtheorem{lem}[thm]{Lemma}
\newtheorem{prop}[thm]{Proposition}

\theoremstyle{remark}
\newtheorem{rem}{Remark}[section]

\numberwithin{equation}{section}

\newcommand{\fg}{{\mathfrak g}}
\newcommand{\fh}{{\mathfrak h}}
\newcommand{\fk}{{\mathfrak k}}
\newcommand{\fl}{{\mathfrak l}}
\newcommand{\fn}{{\mathfrak n}}
\newcommand{\fs}{{\mathfrak s}}
\newcommand{\fgl}{{\mathfrak{gl}}}
\newcommand{\fsl}{{\mathfrak{sl}}}
\newcommand{\fso}{{\mathfrak{so}}}
\newcommand{\fsp}{{\mathfrak{sp}}}

\renewcommand{\L}{{\text{L}}}

\newcommand{\bbC}{{\mathbb C}}
\newcommand{\bbN}{{\mathbb N}}
\newcommand{\la}{\lambda}

\newcommand{\mult}{{\operatorname{mult}}}
\newcommand{\Hom }{{\operatorname{Hom }}}
\newcommand{\End }{{\operatorname{End }}}
\newcommand{\Res }{{\operatorname{Res }}}
\renewcommand{\S   }{{\operatorname{S   }}}

\newcommand{\GL}{{\text{GL}}}
\newcommand{\SL}{{\text{SL}}}
\newcommand{\SO}{{\text{SO}}}
\newcommand{\Sp}{{\text{Sp}}}
\newcommand{\G}{{\text{G}}}
\newcommand{\K}{{\text{K}}}
\newcommand{\M}{{\text{M}}}
\newcommand{\F}{{\text{F}}}
\renewcommand{\S}{{\text{S}}}
\newcommand{\UK}{{\text{U}_{\text{K}}}}
\newcommand{\US}{{\text{U}_{\text{S}}}}
\newcommand{\TK}{{\text{T}_{\text{K}}}}
\newcommand{\TS}{{\text{T}_{\text{S}}}}

\begin{document}
\title[Lowest principal sl(2)-types]{
Lowest $\fsl(2)$-types in $\fsl(n)$-representations \\
        with respect to a principal embedding}
\author[H.~Lhou]{Hassan Lhou}
\address[Hassan Lhou]{Department of Mathematical Sciences\\
         University of Wisconsin - Milwaukee\\
         3200 North Cramer Street\\
         Milwaukee, WI 53211}
\email[H.~Lhou]{hlhou@uwm.edu}
\author[J.~F.~Willenbring]{Jeb F. Willenbring}
\address[Jeb F.~Willenbring.]{Department of Mathematical Sciences\\
         University of Wisconsin - Milwaukee\\
         3200 North Cramer Street \\
         Milwaukee, WI 53211}
\email[J.~F.~Willenbring]{jw@uwm.edu}
\thanks{The second author was supported by the National Security Agency grant \# H98230-09-0054.}

\date{\today}


\abstract{Fix $n \geq 3$.  Let $\fs$ be a principally embedded $\fsl_2$}-subalgebra in $\fsl_n$.  A special case of results of the second author and Gregg Zuckerman implies that there exists a positive integer $b(n)$ such that for any finite dimensional irreducible $\fsl_n$-representation, $V$, there exists an irreducible $\fs$-representation embedding in $V$ with dimension at most $b(n)$.  We prove that $b(n)=n$ is the sharpest possible bound.  We also address embeddings other than the principal one.

The exposition involves an application of the Cartan--Helgason theorem, Pieri rules, Hermite reciprocity, and a calculation in the ``branching algebra'' introduced by Roger Howe, Eng-Chye Tan, and the second author.
}
\maketitle

\section{Introduction}\label{sec_Introduction}

For a positive integer $n=d+1$, the complex irreducible representation of $\fsl_2=\fsl_2(\bbC)$ with dimension $n$ will be denoted $\F_d$.  By definition, the action of $\fsl_2$ on $\F_d$ defines a homomorphism
\[
    \pi: \fsl_2 \rightarrow \End(\F_d).
\]

Furthermore, upon fixing an ordered basis for $\F_d$ we obtain an identification $\End(\F_d) \cong \fgl_n$.  The kernel of $\pi$ is trivial since $\fsl_2$ is a simple Lie algebra.  The image of $\pi$, denoted $\fs$, is therefore isomorphic to $\fsl_2$.  We will refer to $\fs$ as a \emph{principal} $\fsl_2$-subalgebra of $\fgl_n$.

Once again we observe that $\fs$ is simple, and therefore intersects the center of $\fgl_n$ trivially.  Note that here we are require $n>1$.  Hence, $\fs \subseteq \fsl_n$.

The embeddings of $\fsl_2$ into a Lie algebra is fundamental to representation theory.
In general, a subalgebra $\fl$ of a semisimple Lie algebra $\fg$ is said to be a \emph{principal $\fsl_2$-subalgebra} if $\fl \cong \fsl_2$ and it contains a
regular nilpotent element of $\fg$ (\cite{Dynkin-semisimple}, \cite{GOV},
\cite{Kostant}).  These subalgebras are conjugate, so we
sometimes speak of ``the'' principal $\fsl_2$-subalgebra.

There is a beautiful connection between the principal $\fsl_2$-subalgebra and the (co)homology of the corresponding connected, simply connected, compact Lie group (see \cite{Kostant} and the exposition in \cite{CM}).  One considers the decomposition of the adjoint representation when restricted to a principal $\fsl_2$.  In the special case of $\fsl_n$, one obtains that the group of special unitary matrices,  $\text{SU}(n)$, has the same homology as the product of spheres of dimension
$3, 5, 7, \cdots, 2n-1$.

In light of this fact, one considers decomposing representations other than the adjoint.  Even in the case of $\fsl_n$, this is a difficult problem.  Indeed, included in the representations of $\fsl_n$ are the symmetric powers of $\F_d$.  An ``explicit`` decomposition was the subject of nineteenth century invariant theory of binary forms (see  \cite{Howe-bin}).  In the case of the cubic and quartic forms, there is a rich literature (see \cite{VW} and the references therein).
\bigskip
Put another way, composing the principal embedding with an arbitrary finite dimensional representation irreducible representation (irrep. for short) of $\fsl_n$:
\[
    \fsl_2 \hookrightarrow \fsl_n \hookrightarrow \End(\text{V})
\]
is an instance of \emph{plethysm}.  Decomposing V under $\fsl_2$ explicitly is a subject filled with unsolved problems.

Suppose, for example, that
\[
    V = V_0 \oplus V_1 \oplus V_2 \oplus V_3 \oplus \cdots
\] where $V_j$ is the isotypic component of the $\fsl_2$-irrep. $\F_j$.  We can ask,
\emph{When is $V_j \neq (0)$ for an arbitrary finite dimensional irreducible representation ?}  For the ``highest'', that is $\max \{ j : V_j \neq (0) \}$,
 component this question is not difficult since it is the restriction of the highest weight of V to the Cartan subalgebra of $\fsl_2$.  However, the ``lowest'' component (i.e. $\min \{ j : V_j \neq (0) \}$) is more complicated to determine.

This motivates

\begin{thm}\label{thm_main}  Fix $n \geq 3$, and a principal $\fsl_2$-subalgebra, $\fs$, of $\fsl_n$.  Let $V$ denote an arbitrary finite dimensional complex irreducible representation of $\fsl_n$.  Then, there exists $0 \leq d < n$ such that upon restriction to $\fs$, V contains the $\fs$-irrep. $\F_d$ in the decomposition.
\end{thm}

A proof is provided in Section \ref{sec_GeneralCase}, which reduces the problem to the theory of binary forms, which we consider in Section \ref{sec_BinaryForms}.  All notation and convensions are set up in Section \ref{sec_NotConv}.

Proving the general case involves an application of the Cartan-Helgason theorem applied to the orthogonal and symplectic subgroups of the special linear group, together with the Pieri rules for decomposing certain tensor products.  The special case of binary forms is handled using a case-by-case analysis utilizing Hermite reciprocity.  Also, a key ingredient is the fact that the subspace of highest weight vectors in the coordinate ring of a representation is a graded algebra.

It is worth noting that if the embedding of $\fsl_2$ is not \emph{principal} then Theorem \ref{thm_main} still holds.  This fact is elementary but we provide the details in Section \ref{sec_AnyEmb}.

\bigskip

One source of motivation for this result is as follows.  Let $\K$ is a reductive algebraic group with a small subgroup S.  By \emph{small} (see \cite{WZ}) we mean that there exists an integer $b$ such that when an arbitrary irrep. of K is restricted to S there exists an S-irrep. in the decomposition with dimension at most $b$.  If we assume that S is semisimple, then there are only finitely many irreps. with dimension less than $b$.  Thus, the dual of K may be partitioned into equivalence classes in the same sense as done in \cite{BGG} for (K,S) a symmetric pair corresponding to a split real group.  We provide more details concerning motivation in Section \ref{sec_motivation}.

\bigskip

\bigskip
\noindent{\bf Acknowledgement:}

A preliminary report of this work was announced during the Special Session on Analysis and Geometry on Lie Groups organized by Chal Benson and Gail Ratcliff.  This session was part of the Central Fall Sectional Meeting of the American Mathematical Society at the University of Wisconsin-Eau Claire.  The authors thank the organizers of this conference.

\section{Notation and conventions}\label{sec_NotConv}

We record here standard assumptions and notation used in the paper.
The ground field is, $\bbC$, the complex numbers.  Unless explicitly stated, all Lie algebras and representations are assumed to be finite dimensional and over $\bbC$.

If $m$ is a positive integer and $V$ is a vector space then we write
\[
    \begin{array}{llc}
        m V & := & \underbrace{V \oplus V \oplus \cdots \oplus V} \\
            &    & \text{$m$ copies}.
    \end{array}
\]  If $m=0$ then $mV = (0)$.  Thus, $m V = V \otimes \bbC^m$.  Here (and throughout) we are tensoring over the ground field, $\bbC$.

\bigskip

If G is a group (resp. Lie algebra), which acts on a vector space $V$, we will denote the subspace of point-wise fixed vectors (i.e. invariant vectors) by $V^{\G}$.  If G acts on both $V_1$ and $V_2$ then  G acts on $\Hom(V_1,V_2)$ by $g \cdot T:=gTg^{-1}$.  The G-invariant vectors in $\Hom(V_1,V_2)$, are exactly the G-equivariant homomorphisms, denoted $\Hom_{\G}(V_1, V_2)$.  We define:
\[
    \mult_\G(V_1:V_2) := \dim \Hom_{\G}(V_1, V_2).
\]  (If G is understood, it will be omitted.)
Note that in the case that $V_1$ (resp. $V_2$) is irreducible and $V_2$ (resp. $V_1$) is completely reducible then $\mult(V_1, V_2)$ is the \emph{multiplicity} of $V_1$ (resp. $V_2$) in $V_2$ (resp. $V_1$).

Fix $\{ V_\la:  \la \in \widehat{\G} \}$ to be distinct representatives of the equivalence classes of irreducible representations of G, with index set $\widehat{\G}$.  Then, for an arbitrary completely reducible G-representation, $V$,
\[
    V \cong \bigoplus_{\la \in \widehat{G}} m_\la V_\la
\] where the non-negative integers $m_\la$ are the multiplicities, that is
$m_\la = \mult(V_\la : V)$.

If G is a subgroup of a larger group H, and $V$ is an irreducible H-representation then $V$ becomes a G-representation under the restricted action, which we denote by $\Res^\text{H}_\G V$.  Then, the numbers
\[
    \mult_\G( V_\la : \Res^\text{H}_\G V).
\]
are sometimes called the \emph{branching multiplicities}.  An ``explicit'' description of these numbers is called a \emph{branching rule} (or law).
The problem of describing these numbers is connected to the topic of \emph{symmetry breaking} in the physics literature.

\subsection{Lie theoretic setup}

Let $\fk$ denote a rank $r$ simple Lie algebra, and fix a Cartan subalgebra $\fh$.  Denote the root system determined by ($\fk,\fh$) by $\Phi \subset \fh^*$.  For
$\alpha \in \Phi$, let $\fk_\alpha \subseteq \fk$ denote the $\alpha$ root space.

Choose a set of positive roots $\Phi^+$ and set $\fn^+$ be the sum of $\fk_\alpha$ for positive $\alpha$.  The simple roots in $\Phi^+$ will be denoted by
$\Pi = \{\alpha_1, \cdots, \alpha_r \} $.  Let $\omega_1, \cdots, \omega_r$ be the fundamental weights.  The lattice of $\fk$-integral weights is then denoted
$P(\fk)$, while the cone of dominant $\fk$-integral weights is
$P_+(\fk) = \sum_{j=1}^r \bbN \omega_j$ where $\bbN = \{0,1,2,3,\cdots\}$ is the set of non-negative integers.

Set $\mathfrak b = \fh \oplus {\fn}^+$.  For $\la \in P_+(\fk)$ let $\L_{\fk}(\la)$ denote the highest weight representation with highest weight $\la$.

There exists a unique (up to isomorphism) simply connected algebraic group K with Lie algebra $\fk$.  Let $T_K$ denote the maximal algebraic torus in $K$ with Lie algebra $\fh$ and let $U_K$ denote the maximal unipotent group with Lie algebra $\fn^+$.

The irreducible finite dimensional representations of K$\times$K are, up to equivalence, of the form $V_1 \otimes V_2$ where $V_j$ is an irreducible representation of K (with the standard action of K$\times$K on tensors).

Let $\bbC[\K]$ denote the coordinate ring of regular functions on $K$.  The group $K$ acts on itself by left and right multiplication.  Thus, $\bbC[\K]$ is an infinite dimensional representation of K$\times$K with respect to the action defined by $[(k_l, k_r) \cdot f](x) = f(k_l^{-1}xk_r)$.  We have the (algebraic version of the) Peter-Weyl decomposition:
\[
    \bbC[\K] \cong \bigoplus_{\la \in P_+(\fk)} L_K(\la)^* \otimes L_K(\la).
\]

\subsection{Small subalgebras}

Let $\fs$ denote a simple subalgebra of $\fk$.  As we did for $\fk$, we fix a Cartan subalgebra for $\fs$, as well as a choice of positive roots.  Let $P_+(\fs)$ denote the cone of dominant $\fs$-integral weights.  For $\mu \in P_+(\fs)$, we denote the (irreducible finite dimensional) highest weight $\fs$-representation by  $\L_{\fs}(\mu)$.

For $\la \in P_+(\fk)$ define
\[
    \min_\fs \dim(\la) := \min \{ \dim \L_\fs(\mu) :  \mu \in P_+(\fs) \mbox{ and }
    \mult( L_\fs(\mu) : L_\fk(\la))>0 \}.
\]  If the Lie algebra $\fs$ is understood then we will simply write $\min \dim(\la)$.

We wish to compute
\[
    b(\fk,\fs) := \max_{\la \in P_+(\fk)} \min_\fs \dim (\la)
\]
provided a maximum exists.  If no maximum exists then we write
$b(\fk,\fs) = \infty$.  In the former case ($b(\fk,\fs) < \infty$) we say that $\fs$ is \emph{small}\footnote{This notation was introduced in \cite{WZ}.} in $\fk$.  When $\fs$ is small in $\fk$ then we would like to know a sharp bound.

\subsection{A branching algebra}

There exists a connected, Zariski closed, subgroup, S, of K with Lie algebra $\fs$.
Let $\TS$ (resp. $\US$) denote the corresponding maximal torus (resp. unipotent subgroup) in $S$.

We restrict the action of $\K \times \K$ on $\bbC[\K]$ to $S \times K$.  That is, we restrict the left translation by K to that of S.  Define:
\[
    \mathcal B^\K_\S := \bbC[\K]^{\US \times \UK}.
\]
For each of the groups S and K, the maximal torus normalizes the unipotent group.  Thus,
$\TS \times \TK$ acts on these unipotent invariants.  In this light, we obtain a gradation by the lattice cone $P_+(\fs) \times P_+(\fk)$.  Indeed, $\mathcal B^\K_\S$ consists of the highest weight vectors for the $\S \times \K$-action on $\bbC[\K]$.  Let
$\mathcal B^{\K,\la}_{\S,\mu}$ denote the $T_S \times T_K$-eigenspace.  We have
\[
    \mathcal B^{\K}_{\S} = \bigoplus \mathcal B^{\K,\la}_{\S,\mu}
\] where the sum is over $(\la,\mu) \in P_+(\fs) \times P_+(\fk)$.  This is an algebra gradation.

\begin{prop}\label{prop_branch}  Let $\la_1, \la_2 \in P_+(\fk)$ and $\mu_1, \mu_2 \in P_+(\fs)$.  If for $j = 1$ and $j=2$,
\[
    \mult_\S( \L_\S(\mu_j): \Res^\K_\S \L_\K(\la_j) ) > 0,
\]
then
\[
    \mult_\S( \L(\mu_1+\mu_2): \Res^\K_\S \L(\la_1 + \la_2) ) > 0,
\]
\end{prop}
\begin{proof}
The dimension of the graded components of the branching algebra are equal to the branching multiplicities.  Since K is connected, the branching algebra is a subalgebra of the integral domain $\bbC[\K]$, and therefore has no zero divisors.
\end{proof}

It is worth pointing out that the proposition is particularly useful when we write $\mu_2 = \mu$ and consider the special case $\mu_2 = 0$.

\begin{cor}\label{cor_branch}  For $\la_1, \la_2 \in P_+(\fk)$ and $\mu \in P_+(\fs)$.
Assume $\L_\K(\la_2)^\S \neq (0)$.  Then if $\L_\S(\mu)$ occurs in $\L_\K(\la_1)$ as an S-representation, then $\L_S(\mu)$ occurs in
\[
    \L_K(\la_1 + \la_2), \;\; \L_K(\la_1 + 2\la_2), \;\; \L_K(\la_1 + 3\la_2), \;\; \cdots
\]
\end{cor}

\section{Binary forms}\label{sec_BinaryForms}

We now turn our attention to the case when $\fk = \fsl_n$ for $n \geq 3$ and $\fs$ is a principally embedded $\fsl_2$-subalgebra.  We order the fundamental weights $\omega_1, \cdots, \omega_{n-1}$ so that the highest weight of the irreducible representation $\wedge^k \bbC^n$ is $\omega_k$.

Our goal of this article is to prove that
\[
    b(\fsl_n,\fs)  = n.
\]
It is clear that this number is at least $n$ since $\L(\omega_1) \cong \bbC^n$ (the defining representation of $\fsl_n$), and upon restriction to $\fs \cong \fsl_2$ we have $\L(\omega_1) \cong \F_d$ with $\dim \F_d = d+1 = n$.  Recall that $\F_d \cong \S^d(\bbC^2)$ (the degree $d$ symmetric tensors on $\bbC^2$).

As a first step, we consider the problem of computing $\min_\fs \dim (m \omega_1)$ for $d\geq 1$.  As a representation of $\fsl_n$, the space of degree $m$ symmetric tensors, $\S^m(\bbC^n)$, is equivalent to $\L(m \omega_1)$.  Our question is therefore:  What is the minimal dimension of an irreducible $\fsl_2$-representation occurring in the decomposition of
\[
    \S^m[\S^d(\bbC^2)] \cong a_1 \F_1 \oplus a_2 \F_2 \oplus a_3 \F_3 \oplus \cdots.
\]
Here $a_j$ is the multiplicity of $\F_j$.  We need to compute the \emph{lowest} $j$ with a positive multiplicity.  That is, define
\[
    \ell(m,d) := \min \left\{j : a_j > 0 \right\}
\]
for non-negative $m$ and $d$.  Note that we use the letter $\ell$ as it is short from \emph{lowest} $\fsl_2$-type.

\subsection{The $\fsl_2$-character of a  symmetric power}

The representation, $\pi :\SL_2 \rightarrow \text{GL}(V)$ has character
\[
    \begin{array}{cccc}
        \chi_V : &  \text{SL}_2 & \rightarrow & \bbC \\
                &  g & \mapsto & \text{Trace} \pi(g).
    \end{array}
\]  Henceforth, we will let $\chi_V(q)$ denote the value of the character on the diagonal matrix $\left[ \begin{array}{cc} q & 0 \\ 0 & q^{-1} \end{array} \right]$.
Furthermore, we let $\chi_d$ denote the character of $V=\F_d$.  Observe that from the theory of $\fsl_2$-representation we have
\[
    \chi_d(q) = q^{-d} + q^{-d+2} + \cdots + q^{d-2} + q^d.
\]

Consequently, in the case that $V = \S^m[\S^d(\bbC^2)]$, the character is the coefficient of $t^m$ in the Taylor series expansion of
\[
    \frac{1}{(1-t q^d) (1-t q^{d-2}) \cdots (1-t q^{-d})}.
\]
around $t=0$.  As is often done, we will denote this coefficient
$\binom{m+d}{d}_q$,
since the value at $q=1$ is the corresponding binomial coefficient.

The problem of decomposing $\S^m[\S^d(\bbC^2)]$ is equivalent to expanding
\[
\binom{m+d}{m}_q = a_1 \chi_1 + a_2 \chi_2 + a_3 \chi_3 + \cdots
\]
which for fixed $m$ and $d$ can easily be done.  Once the expansion is given, it is easy to obtain $\min_{a_i > 0} i$.

In the following matrix\footnote{Computed in MAPLE} we record $\ell(m,d)$ in the row $m = 0,1,2,\cdots,16$ and column $d = 0,1,2,3,\cdots,16$.
\[
\begin{array}{c||ccccccccccccccccc} \ell
  & 0& 1& 2& 3& 4& 5& 6& 7& 8& 9&10&11&12&13&14&15& \cdots \\ \hline \hline
 0& 0& 0& 0& 0& 0& 0& 0& 0& 0& 0& 0& 0& 0& 0& 0& 0& \cdots \\
 1& 0& 1& 2& 3& 4& 5& 6& 7& 8& 9&10&11&12&13&14&15& \cdots \\
 2& 0& 2& 0&2& 0&2& 0&2& 0&2& 0&2&0&2&0&2& \cdots \\
 3& 0& 3&2&3&0&3&2&3&0&3&2&3&0&3&2&3& \cdots \\
 4& 0& 4&0&0&0&0&0&0&0&0&0&0&0&0&0&0& \cdots \\
 5& 0& 5&2&3&0&1&2&1&0&1&2&1&0&1&2&1& \cdots \\
 6& 0& 6&0&2&0&2&0&2&0&2&0&2&0&2&0&0& \cdots \\
 7& 0& 7&2&3&0&1&2&1&0&1&2&1&0&1&0&1& \cdots \\
 8& 0& 8&0&0&0&0&0&0&0&0&0&0&0&0&0&0& \cdots \\
 9& 0& 9&2&3&0&1&2&1&0&1&0&1&0&1&0&1& \cdots \\
10& 0&10&0&2&0&2&0&2&0&0&0&0&0&0&0&0& \cdots \\
11& 0&11&2&3&0&1&2&1&0&1&0&1&0&1&0&1& \cdots \\
12& 0&12&0&0&0&0&0&0&0&0&0&0&0&0&0&0& \cdots \\
13& 0&13&2&3&0&1&2&1&0&1&0&1&0&1&0&1& \cdots \\
14& 0&14&0&2&0&2&0&0&0&0&0&0&0&0&0&0& \cdots \\
15& 0&15&2&3&0&1&0&1&0&1&0&1&0&1&0&1& \cdots \\
\vdots & \vdots & \vdots & \vdots & \vdots & \vdots &
\vdots & \vdots & \vdots & \vdots & \vdots & \vdots &
\vdots & \vdots & \vdots & \vdots & \vdots & \ddots
\end{array}
\]

A first observation is that the above matrix is symmetric.  An important tool for our solution to this problem is the \emph{Hermite Reciprocity} (see \cite{GW})
\begin{thm}
As $\fsl_2$-representations,
\[
    \S^m[\S^d(\bbC^2)] \cong \S^d[\S^m(\bbC^2)]
\] for all $m, d \in \bbN$.
\end{thm}
\begin{proof} The result follows from the fact that the two sides have the same character. That is, for any non-negative $m$ and $d$,
\[
    \left[ \begin{array}{c} m + d \\ d \end{array} \right]_q =
    \left[ \begin{array}{c} d + m \\ m \end{array} \right]_q.
\]
\end{proof}

Apart from Hermite reciprocity, there are many ``obvious'' aspects of this matrix. \begin{prop}\label{prop_odd}
If both $m$ and $d$ are odd then $\ell(m,d) \geq 1$,
\end{prop}
\begin{proof}
The weights of $\S^m[\S^d(\bbC^2)]$ are of the form
\[
    m_d d + m_{d-2} (d-2) + \cdots + m_{-d} (-d)
\] where $m_d + \cdots + m_{-d} = m$.  Reduce modulo 2 to see that zero is not a weight when $m$ and $d$ are both odd.  If a non-zero $\fsl_2$-invariant exists, then the zero weight space is non-trivial.
\end{proof}

Also, the zeroth row (resp. column) is 0 since the symmetric powers of the trivial representation, $\F_0$, are trivial.  The next row (resp. column) is given by $0,1,2,3,\cdots$ since the linear forms on $\F_d$ are precisely $\F_d$.

The subsequent rows (resp. columns) corresponding to the symmetric tensors on $\F_d$ for $2 \leq d \leq 7$ can be described using the branching algebra.  That is, we compute an initial segment of each row and then apply Corollary \ref{cor_branch} once we find an invariant.  We proceed on a case-by-case basis.

\subsection{Quadratic forms}

\begin{lem}\label{lem_d2}
    For any $d \in \bbN$,
\[ \ell(2,d) =
\left\{
  \begin{array}{ll}
    0, & \hbox{$d$ even;} \\
    2, & \hbox{$d$ odd.}
  \end{array}
\right.
\]
\end{lem}
\begin{proof}
Row 2 (i.e. third from the top) consists of $0,2,0,2,0,2,\cdots$.  The fact that the first three entries are $0,1,0$ is a straightforward calculation.  We then observe that $\S^{2k}[\S^2(\bbC^2)]$ contains an invariant for all $k$ by taking powers of the invariant in the $k=1$ case.  Then, multiplying each of these invariants by the highest weight vector in $\S^1[\S^2(\bbC^2)]$.  Proposition \ref{prop_odd} implies that there are only invariants degree $2k$.
\end{proof}

\subsection{Cubic forms}
\begin{lem}\label{lem_d3}
For any $d \in \bbN$,
\[ \ell(3,d)=
\left\{
  \begin{array}{ll}
    0, & \hbox{$d \equiv 0 \mod 4$;} \\
    2, & \hbox{$d \equiv 2 \mod 4$;} \\
    3, & \hbox{$d \equiv \pm 1 \mod 4$.} \\
  \end{array}
\right.
\]
\end{lem}
\begin{proof}
Row 3 (i.e. forth from the top) consists of the repeating pattern $0,3,2,3$.  This fact follows from the fact that $\S^4[\S^3(\bbC^2)]$ has a non-zero invariant vector (a straightforward calculation).  Multiplying its powers by the highest weight vectors occurring in degree at most three yields the repeated occurrence of $\F^0, \F^3, \F^2, \F^3$.  The fact that no irreps. of lower dimension occur follows from the theory of the cubic (see \cite{VW}, and the references therein).
\end{proof}

We should note that it is well known that the full algebra of invariants on the cubic is generated by the degree four invariant.  Furthermore, the functions on the cubic are a free module over the invariant subalgebra (see \cite{GW}, Chapter 12).  The fact that no lower $\F_d$'s appear other than the repeating patter $0,3,2,3$ is an immediate consequence of the story for the cubic.

\subsection{Quartic forms}
\begin{lem}\label{lem_d4}
$\ell(4,1)=4$ and for all $d\neq 1$ we have $\ell(4,d)=0$.
\end{lem}
\begin{proof}
Row 4 (i.e. fifth from the top) has its first four entries 0,4,0,0.  This means that there is an invariant in degree 2 and 3.  Powers and products of these two invariants imply the existence of invariants in degree $d$ for $d \geq 2$.  Thus the remaining entries of the entire row are all zero.
\end{proof}

Note that by Hermite reciprocity we have that $\S^4(\S^d(\bbC^2))$ has an invariant for all $d \neq 1$.  For each $d \geq 0$ fix a choice of non-zero invariant, $f_4(d)$.  We will use these in the remaining cases.

We should note that the invariant polynomials on the quartic form a polynomial ring generated by two polynomials of degree 2 and 3.  The full polynomial algebra is a free module over the invariants -- a consequence of Kostant-Rallis theory (see \cite{GW} Chapter 12).

\subsection{Quintic forms}
\begin{lem}\label{lem_d5}
For odd $d \geq 5$, $\ell(5,d) = 1$, and $\ell(5,1)=5$, $\ell(5,3)=3$.

For even $d \geq 16$, $\ell(5,d)=0$ and
$\ell(5,0)=\ell(5,4)=\ell(5,8)=\ell(5,12)=0$.

Lastly, $\ell(5,2)=\ell(5,6)=\ell(5,10)=\ell(5,14)=2$.
\end{lem}
\begin{proof}
Row 5 (sixth from the top) requires examining columns 0 through 18:
\[
    0, 5, 2, 3, 0, 1, 2, 1, 0, 1, 2, 1, 0, 1, 2, 1, 0, 1, 0.
\]
Since we know we have a degree four invariant, $f_4(5)$, for the quintic we can predict that the final $1,0,1,0$ will continue to repeat by multiplying by powers.  Proposition \ref{prop_odd} guarantees that there are no invariants in the odd degree columns.
\end{proof}

\subsection{Sextic forms}
\begin{lem}\label{lem_d6}  For $d \geq 14$ or any even $d$, $\ell(6,d) = 0$.

For $d=3,5,7,9,11,13$ we have $\ell(6,d)=2$.  Lastly, $\ell(6,1)=6$.
\end{lem}
\begin{proof}
Row 6 (seventh from the top) requires examining columns 0 through 17:
\[
    0, 6, 0, 2, 0, 2, 0, 2, 0, 2, 0, 2, 0, 2, 0, 0, 0, 0.
\]
Since the last four entries are 0, we can infer that the remaining entries are 0 by multiplying by $f_4(6)$.
\end{proof}

\subsection{Septic forms}
\begin{lem}\label{lem_d7}  For odd $d \geq 5$, $\ell(7,d)=1$.  $\ell(7,1)=7.$
$\ell(7,3)=3$.  $\ell(7,2)=\ell(7,6)=\ell(7,10)=2$.
For even $d \neq 2,6,10$, $\ell(7,d)=0$.
\end{lem}
\begin{proof}
Row 7 (eighth from the top) requires examining columns 0 through 14:
\[
0, 7, 2, 3, 0, 1, 2, 1, 0, 1, 2, 1, 0, 1, 0.
\]
Since the last four entries are $1,0,1,0$ we can expect this pattern to repeat by multiplying by $f_4(7)$ and observing that there are no invariants in odd degree (again my Proposition \ref{prop_odd}).
\end{proof}

\subsection{Forms of arbitrary degree.}

A final observation is that most of the interior of the table is filled with zeros and ones, which we prove now.   With this in mind define:

\[
E = \left\{ (5,6),(5,10),(5,14),(6,7),(6,9),(6,11),(6,13),(7,10) \right\}.
\]

We have
\begin{prop}  The values of $\ell(d,m)$ are as follows.
\begin{itemize}
\item For all $m \in \bbN$, $\ell(m,0)=0$.
\item For all $m \in \bbN$, $\ell(m,1)=m$.
\item For all $m \in \bbN$,
    \[  \ell(m,2) =
        \left\{
        \begin{array}{ll}
            0, & \hbox{$m$ even;} \\
            2, & \hbox{$m$ odd.}
        \end{array}
        \right.
    \]
\item For all $m \in \bbN$,
    \[  \ell(m,3)=
        \left\{
        \begin{array}{ll}
            0, & \hbox{$m \equiv 0 \mod 4$;} \\
            2, & \hbox{$m \equiv 2 \mod 4$;} \\
            3, & \hbox{$m \equiv \pm 1 \mod 4$.} \\
        \end{array}
        \right.
    \]
\item For all $m \in \bbN$, if $m \neq 1$ then $\ell(4,m)=0$; $\ell(4,1)=4$.
\item For all $d \geq m \geq 5$
\[ \ell(m,d) =
\left\{
  \begin{array}{ll}
    2, & \hbox{$(m,d) \in    E$;} \\
    1, & \hbox{$(m,d) \notin E$ and $m \equiv d \equiv 1 \mod 2$;} \\
    0, & \hbox{otherwise.}
  \end{array}
\right.
\]
\end{itemize}
\end{prop}
\begin{proof}  For $d,m=0,1$ the result is trivial.
For $d,m=2,3,4,5,6,7$ the result follows from Lemmas
\ref{lem_d2}, \ref{lem_d3}, \ref{lem_d4}, \ref{lem_d5}, \ref{lem_d6}, \ref{lem_d7} and Hermite reciprocity.
If $d,m \geq 8$ then we revisit the table of values for $\ell$:
\[
\begin{array}{c||cccccccc|cccc|ccccc} \ell
  & 0& 1& 2& 3& 4& 5& 6& 7& 8& 9&10&11&12&13&14&15& \cdots \\ \hline \hline
 0& 0& 0& 0& 0& 0& 0& 0& 0& 0& 0& 0& 0& 0& 0& 0& 0& \cdots \\
 1& 0& 1& 2& 3& 4& 5& 6& 7& 8& 9&10&11&12&13&14&15& \cdots \\
 2& 0& 2& 0&2& 0&2& 0&2& 0&2& 0&2&0&2&0&2& \cdots \\
 3& 0& 3&2&3&0&3&2&3&0&3&2&3&0&3&2&3& \cdots \\
 4& 0& 4&0&0&0&0&0&0&0&0&0&0&0&0&0&0& \cdots \\
 5& 0& 5&2&3&0&1&2&1&0&1&2&1&0&1&2&1& \cdots \\
 6& 0& 6&0&2&0&2&0&2&0&2&0&2&0&2&0&0& \cdots \\
 7& 0& 7&2&3&0&1&2&1&0&1&2&1&0&1&0&1& \cdots \\ \hline
 8& 0& 8&0&0&0&0&0&0&0&0&0&0&0&0&0&0& \cdots \\
 9& 0& 9&2&3&0&1&2&1&0&1&0&1&0&1&0&1& \cdots \\
10& 0&10&0&2&0&2&0&2&0&0&0&0&0&0&0&0& \cdots \\
11& 0&11&2&3&0&1&2&1&0&1&0&1&0&1&0&1& \cdots \\ \hline
12& 0&12&0&0&0&0&0&0&0&0&0&0&0&0&0&0& \cdots \\
13& 0&13&2&3&0&1&2&1&0&1&0&1&0&1&0&1& \cdots \\
14& 0&14&0&2&0&2&0&0&0&0&0&0&0&0&0&0& \cdots \\
15& 0&15&2&3&0&1&0&1&0&1&0&1&0&1&0&1& \cdots \\
\vdots & \vdots & \vdots & \vdots & \vdots & \vdots &
\vdots & \vdots & \vdots & \vdots & \vdots & \vdots &
\vdots & \vdots & \vdots & \vdots & \vdots & \ddots
\end{array}
\]
and observe the, eventual, periodicity modulo 4 along all rows and columns.

Specifically, rows and columns 8 through 11 form a $4 \times 4$ matrix:
\[
\left[ \begin {array}{cccc} 0&0&0&0\\\noalign{\medskip}0&1&0&1\\\noalign{\medskip}0&0&0&0\\\noalign{\medskip}0&1&0&1\end {array}
 \right]
\]
The highest weight vectors corresponding to the weights in the above matrix may be multiplied by the degree four invariants $f_4(8),f_4(9),f_4(10),f_4(11)$ to translate the $4 \times 4$-submatrix to the right.  Then, by Hermite reciprocity, these four rows may be translated down by multiplying by $f_4(m)$ for $m \geq 8$.
\end{proof}

From the point of view of proving our main theorem we need
\begin{cor}\label{cor_BinaryForms}
    For all $m\geq 0$ and $d \geq 2$, $\ell(m,d) \leq d$.
\end{cor}

That is to say, for any $m$ and $d \geq 1$, there exists an irreducible representation of $\fsl_2$ with dimension at most $d+1$ occurring in $\S^m[\S^d(\bbC^2)]$.  This completes our analysis of the lowest $\fsl_2$-types occurring in binary forms.

\section{The proof of the general case}\label{sec_GeneralCase}

A principal $\fsl_2$ subalgebra of $\fsl_n$ is given by the image of $\fsl_2$ in the $n$-dimensional irreducible representation of $\fsl_2$ (over $\bbC$).  For $n \geq 4$ these are not maximal.  Rather, for even $n$ the image is contained in the standard symplectic subalgebra of $\fsl_n$, while for odd $n$ the image is contained in the standard orthogonal subalgebra of $\fsl_n$.  Because of this fact, it is natural to approach the general situation the two cases (Sections \ref{subsec_even} and \ref{subsec_odd}) corresponding to the parity of $n$.  The argument is a fairly straightforward application of the Cartan-Helgason theorem (Section \ref{subsec_CH}) and the Pieri rules (Section \ref{subsec_Pieri}) to reduce the situation to binary forms.  The latter have already been analyzed in Section \ref{sec_BinaryForms}.

\subsection{The Cartan--Helgason theorem}\label{subsec_CH}

Let (K,M) denote a symmetric pair.  That is, M is the fixed point set of an algebraic group involution of K.  Then the pair is spherical, which means here that the affine variety $\K/\M$ has a multiplicity free coordinate ring.  Put another way, for any irreducible representation, $V$, of K, the dimension of the M-invariant subspace, $V^\text{M}$, is at most one dimensional.   These results are part of the Cartan--Helgason theorem (see \cite{GW}).

For our purposes, we consider two examples when $\K = \SL_n$.  First $\M = \SO_n$.  Although the pair is symmetric for all $n$,  we will only consider this example for $n$ odd.  For the second case, we will let $\M=\Sp_n$, for $n$ even.  We have
\begin{prop}\label{prop_CH_SO}
The pair ($\SL_n, \SO_n$) is symmetric and for $\la \in P_+(\fsl_n)$ we have

\[ \dim L(\la)^{\SO_n} = 1 \]
if and only if
\[
    \la \in 2\bbN \omega_1 \oplus 2\bbN \omega_2 \oplus 2\bbN \omega_3 \oplus \cdots \oplus 2\bbN \omega_{n-1}.
\]
That is, the Young diagram corresponding to $\la$ has rows with an even number of boxes.
\end{prop}
\begin{proof} See \cite{GW} Chapter 12.
\end{proof}

\begin{prop}\label{prop_CH_Sp}
The pair ($\SL_n, \Sp_n$) is symmetric and for $\la \in P_+(\fsl_n)$ we have
\[ \dim \L(\la)^{\Sp_n} = 1 \]
if and only if
\[
    \la \in \bbN \omega_2 \oplus \bbN \omega_4 \oplus \bbN \omega_6 \oplus \cdots \oplus \bbN \omega_{n-2}.
\]
That is, the Young diagram corresponding to $\la$ has columns with an even number of boxes.
\end{prop}
\begin{proof} See \cite{GW} Chapter 12.
\end{proof}

\subsection{Pieri rules}\label{subsec_Pieri}

Decomposing tensor products of finite dimensional $\GL_n$-irreps. is a well studied area of combinatorial representation theory.  In the special case where one of the tensor factors is a symmetric or exterior power of the defining representation, the results are particularly nice and will be used here.  They are often referred to as \emph{Pieri rules} in the literature.

We choose the maximal torus for $\GL_n$ to be the group of diagonal matrices with non-zero entries, and fix the Borel to be upper triangular matrices.  On the corresponding Cartan subalgebra, we let $\epsilon_i$ denote the linear functional
\[ \epsilon_i:
\left[
  \begin{array}{ccc}
    d_1 &        &     \\
        & \ddots &     \\
        &        & d_n
  \end{array}
\right] \mapsto d_i.
\]  We will call these the standard coordinates.

Then, for a weakly decreasing $n$-tuple of non-negative integers
($\la_1 \geq \cdots \geq \la_n$), we let
\[
    \la := \la_1 \epsilon_1 + \cdots + \la_n \epsilon_n \in P_+(\fgl_n)
\]
denote the corresponding dominant integral highest weight representation.  Note that $\L(\la)$ has polynomial matrix coefficients.  Each $\L(\la)$ restricts irreducibly to $\SL_n$.  Furthermore, each irrep. of $\SL_n$ is the restriction of such an $\L(\la)$.

The $n$ tuple for $\la$, in the standard coordinates, corresponds to a Young diagram whose $j$-th row has $\la_j$ boxes.  We depict the rows of Young diagrams from top to bottom.
So for example, when $n=8$, $(7,5,5,2,0,0,0,0)$ corresponds to
\begin{equation}\label{eq_la}
\begin{Young}
&&&&&&\cr
&&&&\cr
&&&&\cr
&\cr
\end{Young}
\end{equation}
Denote the weight of the above diagram by $\la$.

Recall that the decomposition of $\L(\mu) \otimes \S^d(\bbC^n)$ into irreducibles is multiplicity free with the $\la$ occurring corresponding to Young diagrams obtained by adding $d$ boxes to the Young diagram of $\mu$ with no two in the same column.  For example, when $d=5$ the diagram in (\ref{eq_la}) is obtained from ($5,5,2,2,0,0,0,0$) by adding five boxes:
\begin{equation}\label{eq_mu}
\begin{Young}
&&&&&$X$&$X$\cr
&&&&\cr
&&$X$&$X$&$X$\cr
&\cr
\end{Young}.
\end{equation}
Thus, if we denote the weight of the diagram without the X's by $\mu$ then we know that $\L(\la)$ is one of the irreducible constituents of $\L(\mu) \otimes \S^5(\bbC^8)$, as an $\SL_8$ representation.

Recall that the symmetric powers of the defining representation of symplectic group are irreducible.  Furthermore, it is worth noting that it is always possible to add or remove a ``horizontal strip'' in this way so that the resulting diagram has even columns.
A consequence of this observation is
\begin{prop}\label{prop_sp}
Let $n \geq 2$ be even.  For any irreducible representation, $\L(\mu)$, of $\SL_n$ there exists $d$ such that
\[
    \mult(\S^d(\bbC^n): \Res^{\SL_n}_{\Sp_n} \L(\mu)) = 1.
\]
\end{prop}
\begin{proof}  For an irreducible $\SL_n$-representation $\L(\mu)$ we have
\[
    \Hom_{\Sp_n}( \L(\mu), \S^d(\bbC^n) ) \cong
    \left[ \S^d(\bbC^n) \otimes \L(\mu)^* \right]^{\Sp_n}.
\]  Although finite dimensional $\Sp_n$-irreps. are self dual, this is not true for $\SL_n$-irreps.  However, without loss of generality we can replace the diagram for $\mu$ by the diagram for the dual.  (The Young diagram corresponding to $\L(\mu)^*$ is the complement of the diagram for $\mu$ inside the $\mu_1 \times n$-rectangle, but we do not need this fact.)

Using the Pieri rules we can explicitly decompose the tensor product.  As previously noted, by choosing $d$ appropriately, we can add enough boxes to the Young diagram corresponding to $\mu$ so that the resulting diagram has even columns.  The result follows from Proposition \ref{prop_CH_Sp}.
\end{proof}

\bigskip

Next, recall that the decomposition of, $\L(\nu) \otimes \wedge^d(\bbC^n)$ into irreducibles is multiplicity free with the $\la$ occurring if the corresponding Young diagram is obtained by adding $d$ boxes to the Young diagram of $\nu$ with no two in the same row.  For example, when $d=3$ the diagram in (\ref{eq_la}) is obtained from (6,4,4,2,0,0,0,0) by adding three boxes:

\begin{equation}\label{eq_nu}
\begin{Young}
&&&&&&$X$\cr
&&&&$X$\cr
&&&&$X$\cr
&\cr
\end{Young}
\end{equation}

If we denote the weight of the diagram without the X's by $\nu$ then we know that $\L(\la)$ is one of the irreducible constituents of $\L(\nu) \otimes \wedge^3(\bbC^8)$, as an $\SL_8$ representation.

It is worth noting that it is always possible to add or remove a ``vertical strip'' in this way so that the resulting diagram has even rows.

\begin{prop}\label{prop_so}
For any irreducible representation, $\L(\nu)$, of $\SL_n$ there exists $k$ such that
\[
    \mult\left( \wedge^k(\bbC^n): \Res^{\SL_n}_{\SO_n} \L(\nu) \right) = 1.
\]
Recall that the exterior powers of the defining representation of the orthogonal group are irreducible.
\end{prop}
\begin{proof}
Mutatis-Mutandis from Proposition \ref{prop_sp}.  Change the symplectic group to the (special) orthogonal group.  Change the evenness condition on columns to rows.  Change symmetric powers to exterior powers, and use the appropriate Pieri rule.
\end{proof}

\begin{rem}  The symmetric pair ($\SL_n, \SO_n$) corresponds to the split real form, $\text{GL}_n(\mathbb R)$ and therefore is a special case of \cite{BGG}.  We thank Gregg Zuckerman for reminding us of this reference.
\end{rem}

\subsection{The general case: even $n$}\label{subsec_even}

We reduce the problem for an arbitrary irrep. of $\fsl_n$ to the special case of binary forms.  First, we have the chain of embeddings
\[
    \fsl_2 \hookrightarrow \fsp_n \hookrightarrow \fsl_n.
\]
Upon restriction to $\fsp_n$ we apply Proposition \ref{prop_sp} to prove the existence of an $\fsp_n$-irrep. of the form $\S^m(\bbC^n)$.  We then observe that with respect to our principal $\fsl_2$-subalgebra $\fs$ we have $\bbC^n \cong \F_{d}$ for $n=d+1$.  The result follows for even $n$ from our analysis of the minimal $\fsl_2$-type in binary forms (Section \ref{sec_BinaryForms}).

\subsection{The general case: odd $n$}\label{subsec_odd}

Here, we reduce the problem for an arbitrary irrep. of $\fsl_n$ to the special case of exterior powers of $\fsl_2$-irreps.  First, we have the chain of embeddings
\[
    \fsl_2 \hookrightarrow \fso_n \hookrightarrow \fsl_n.
\]
Upon restriction to $\fso_n$ we apply Proposition \ref{prop_so} for $n$ odd to prove that there exists of an $\fsp_n$-irrep. of the form $\wedge^k(\bbC^n)$.  We then observe that with respect to our principal $\fsl_2$-subalgebra $\fs$ we have $\bbC^n \cong \F_{d}$ for $n=d+1$.

At this point we are faced with the problem of determining the lowest $\fsl_2$-type in $\wedge^k \F_d$

\subsubsection{The $\fsl_2$-character of an exterior power}

We have $n=d+1 \geq 3$.  It is an elementary fact that $\dim \wedge^k[\S^d(\bbC^2)] = \binom{d+1}{k}$,
and we then recognize the above as $\dim \S^k[\S^{(d+1-k)}(\bbC^2)]$.  In fact, the $q$-binomial coefficients are equal
\[
\binom{d+1}{k}_q = \binom{k+(d+1-k)}{k}_q,
\]
which implies
\[ \wedge^k(\F_d) \cong \S^k[\S^{d+1-k}(\bbC^2)].
\]
Clearly, if $k=0$ or $k=d+1$ then $\wedge^k(\F_d)$ is the trivial representation.  Moreover, if  $k=d$ or $k=1$ then $\wedge^k(\F_d) \cong \F_d$.  From Section \ref{sec_BinaryForms}, we know that if $2 \leq k < d$ then $2 \leq d+1-k \leq d-1$, and $\S^k[\S^{d+1-k}(\bbC^2)]$ therefore contains an $\fsl_2$-irrep. $\F_j$ for some $j \leq d+1-k$.  Since $d+1-k \leq d+1$ we are done.

\section{The case of a non-principal embedding}\label{sec_AnyEmb}

As before, let $n \geq 3$.  If $\iota: \fsl_2 \rightarrow \fsl_n$ is a homomorphism of Lie algebras then the defining representation, denoted here as $\bbC^n$, of $\fsl_n$ decomposes into irreps. of $\fsl_2$,
\[
    \bbC^n \cong \F_{d_1} \oplus \cdots \oplus \F_{d_p}
\] with $d_1 \geq \cdots \geq d_p \geq 0$, for some $p \geq 1$.  Let $n_j = d_j+1$ so that
$n=n_1 + \cdots + n_p$.  If all $n_j=1$ then $\iota$ is the zero map, otherwise it is an embedding -- as $\fsl_2$ is simple.

The image, $\iota(\fsl_2)$, is contained in a subalgebra of $\fsl_n$ isomorphic to $\fsl_{n_1} \oplus \cdots \oplus \fsl_{n_p}$.

If $p=1$ then this embedding is principal.  An arbitrary embedding is characterized by the integers $n_1 \geq \cdots \geq n_p$.  That is, we have
\[
    \iota: \fsl_2 \hookrightarrow \fsl_{n_1} \oplus \cdots \oplus \fsl_{n_p} \hookrightarrow \fsl_n.
\]

Upon restriction, an arbitrary finite dimension representation of $\fsl_n$ decomposes into irreps. of $\fsl_{n_1} \oplus \cdots \oplus \fsl_{n_p}$.  Each irreducible summand is equivalent to
\[
    L(\la^{(1)}) \otimes \cdots \otimes L(\la^{(p)})
\] where $L(\la^{(j)})$ is an irrep. of $\fsl_{n_j}$.

Upon further restriction to $\iota(\fsl_2)$ each tensor factor, $L(\la^{(j)})$, contains an irrep., $\F_{k_j}$, of dimension at most $n_j$, by Theorem \ref{thm_main}.  All irreducible $\fsl_2$-representations in the tensor product
\[
    \F_{k_1} \otimes \cdots \otimes \F_{k_p}
\] have dimension at most $n_1 + \cdots + n_p = n$.  The result follows.

\section{A view from the theory of generalized Harish-Chandra modules}\label{sec_motivation}

As was described in the introduction, one source of motivation for us is to generalize the results of \cite{BGG}.  There is another, related, point of view however.  If G is a real reductive group and $\mathcal H$ is an irreducible unitary representation of G, then there exists an underlying admissible Harish-Chandra module, $\mathcal M$, for the pair ($\fg, K$) where $\fg$ is the complexified Lie algebra of G and $K$ is the complexification of a  maximal compact subgroup of G.

The question arises:  Is $\mathcal M$ admissible for the pair ($\fg, S$) where S is small in K? If $\mathcal M$ is infinite dimensional then the answer is {\bf no}.  Thus, a construction of admissible generalized Harish-Chandra modules for ($\fg, S$) produces modules which are not admissible ($\fg, K$)-modules.

This point of view was the original motivation of the authors in \cite{WZ}.  See \cite{PZ} and \cite{PS} for further developments along these lines.  In \cite{WZ} a partial classification of small subalgebras was undertaken without paying attention to obtaining a sharp bound on the dimension of a lowest $\fs$-type.  It would be interesting to obtain bounds in general.

Beyond our methods, note that the geometry of the moment map can be used to approach these questions, see \cite{Webster}.

\end{document}